\documentclass[12pt]{article}
\usepackage{amsmath,amscd,amsthm,a4,amssymb}

\def\Rn{{\mathbb{R}^n}}

\def\i{\infty}

\def\Lplocl2{L^{p,\rm loc}_{w}(l_2)}

\def\Llocl2{L_1^{\rm loc}(l_2)}

\def\WLplocl2{WL_p^{\rm loc}(l_2)}
\def\L1loc{L^{1,\rm loc}_{w}(\Rn)}
\def\L1locl2{L^{1,\rm loc}_{w}(l_2)}

 \newtheorem{thm}{Theorem}[section]
 \newtheorem{cor}[thm]{Corollary}
 \newtheorem{lem}[thm]{Lemma}
 
 \theoremstyle{definition}
 \newtheorem{defn}[thm]{Definition}
 \newtheorem{rem}[thm]{Remark}
 
 \numberwithin{equation}{section}

\usepackage{amssymb}
\usepackage{color}

\def\Rn{{\mathbb{R}^n}}

\def\i{\infty}

\newcommand{\ess}{\mathop{\rm ess \; sup}\limits}
\newcommand{\es}{\mathop{\rm ess \; inf}\limits}
\begin{document}

\begin{center}
\LARGE Commutators of vector-valued intrinsic square functions on vector-valued generalized weighted Morrey spaces
\end{center}

\

\centerline{\large Vagif S. Guliyev, M.N. Omarova}

\

\begin{abstract}

In this paper, we will obtain the strong type and weak type estimates for vector-valued analogues of intrinsic square functions in the generalized weighted Morrey spaces $M^{\Phi,\varphi}_{w}(\Rn)$.
We study the boundedness of intrinsic square functions including the Lusin area integral, Littlewood-Paley $\mathrm{g}$-function and $\mathrm{g}_{\lambda}^{*}$ -function
and their $k$th-order commutators on vector-valued generalized weighted Morrey spaces $M^{\Phi,\varphi}_{w}(l_2)$.
In all the cases the conditions for the boundedness are given either in terms of Zygmund-type integral inequalities on $\varphi(x,r)$ without assuming any  monotonicity property of  $\varphi(x,r)$ on $r$.
\end{abstract}

\



\

\section{Introduction}

It is well-known that the commutator is an important integral operator and it plays a key
role in harmonic analysis. In 1965, Calderon \cite{Cald1, Cald2} studied a kind of
commutators, appearing in Cauchy integral problems of Lip-line.
Let $K$ be a Calder\'{o}n-Zygmund singular integral operator and $b \in BMO(\Rn)$.
A well known result of Coifman, Rochberg and Weiss \cite{CRW} states that the commutator operator
$[b, K] f = K(b f)-b \, Kf$ is bounded on $L^{p}(\Rn)$ for $1 < p < \infty$.
The commutator of Calder\'{o}n-Zygmund operators plays an important role in studying the regularity
of solutions of elliptic partial differential equations of second order (see, for example, \cite{ChFra}-\cite{ChFraL2},
\cite{Chen}, \cite{FazRag2}, \cite{FanLuY}).

The classical Morrey spaces were originally introduced by Morrey in \cite{Morrey} to study the local behavior of solutions to second order elliptic partial differential equations. For the properties and applications of classical Morrey spaces, we refer the readers to \cite{FazRag2, FanLuY, GulJIA, Morrey}.
Recently, Komori and Shirai \cite{KomShir} first defined the weighted Morrey spaces $L^{p,\kappa}(w)$ and studied the boundedness of some classical operators such as the Hardy-Littlewood maximal operator, the Calder\'{o}n-Zygmund operator on these spaces. Also, Guliyev \cite{GulOMTSA, GulEMJ2012} introduced the generalized weighted Morrey spaces $M^{p,\varphi}_{w}$ and studied the boundedness of the sublinear operators and their higher order commutators generated by Calder\'{o}n-Zygmund operators and Riesz potentials in these spaces (see, also \cite{GulKarMustSer, GulMathN, KarGulSer, MustAJM}).

The intrinsic square functions were first introduced by Wilson in \cite{Wilson1, Wilson2}. They are defined as follows. For $0<\alpha\leq 1$, let $C_{\alpha}$ be the family of functions $\phi : \Rn\rightarrow \mathbb{R}$ such that $\phi$'s support is contained in $\{x:|x|\leq 1\}$, $\int_{\Rn} \phi(x) dx =0$, and for $x,\ x'\in \Rn$,
$$
|\phi(x)-\phi(x')|\leq|x-x'|^{\alpha}.
$$
For $(y,\ t)\in \mathbb{R}_{+}^{n+1}$ and $f\in L^{1,loc}(\Rn)$ , set
$$
A_{\alpha}f(t,\ y)\equiv\sup_{\phi\in C_{\alpha}}|f*\phi_{t}(y)|,
$$
where $\displaystyle \phi_{t}(y)=t^{-n}\phi(\frac{y}{t})$ . Then we define the varying-aperture intrinsic square (intrinsic Lusin) function of $f$ by the formula
$$
G_{\alpha,\beta}(f)(x)=\left(\int\int_{\Gamma_{\beta}(x)}(A_{\alpha}f(t,y))^{2}\frac{dydt}{t^{n+1}}\right)^{\frac{1}{2}},
$$
where $\Gamma_{\beta}(x)=\{(y,\ t)\in \mathbb{R}_{+}^{n+1}:|x-y|<\beta t\}$. Denote $G_{\alpha,1}(f)=G_{\alpha}(f)$ .

This function is independent of any particular kernel, such as Poisson kernel. It dominates pointwise the classical square function(Lusin area integral) and its real-variable generalizations. Although the function $G_{\alpha,\beta}(f)$ is depend of kernels with uniform compact support, there is pointwise relation between $G_{\alpha,\beta}(f)$ with different $\beta$:
$$
G_{\alpha,\beta}(f)(x)\leq \beta^{\frac{3n}{2}+\alpha}G_{\alpha}(f)(x)\ .
$$
We can see details in \cite{Wilson1}.

The intrinsic Littlewood-Paley $\mathrm{g}$-function and the intrinsic $\mathrm{g}_{\lambda}^{*}$ function are defined respectively by
$$
\mathrm{g}_{\alpha}f(x)=\left(\int_{0}^{\infty}(A_{\alpha}f(y,t))^{2}\frac{dt}{t}\right)^{\frac{1}{2}},
$$
$$
\mathrm{g}_{\lambda,\alpha}^{*}f(x)=\left(\int\int_{\mathbb{R}_{+}^{n+1}}\left(\frac{t}{t+|x-y|}\right)^{n\lambda}(A_{\alpha}f(y,t))^{2}\frac{dydt}{t^{n+1}}\right)^{\frac{1}{2}}.
$$

When we say that f maps into $l_2$, we mean that $\vec{f}(x)= \big( f_j \big)_{j=1}^{\i}$, where each $f_j$ is Lebesgue measurable and, for almost every $x \in \Rn$
$$
\| \vec{f}(x) \|_{l_2} = \bigg(\sum_{j=1}^{\infty}|f_{j}(x)|^{2}\bigg)^{1/2}.
$$

Let $\vec{f}= (f_{1},\ f_{2},\ \ldots)$ be a sequence of locally integrable functions on $\Rn$. For any $x\in \Rn$, Wilson \cite{Wilson2} also defined the vector-valued intrinsic square functions of $\vec{f}$ by $\| G_{\alpha} \vec{f}(x) \|_{l_2}$ and proved the following result.

\textbf{Theorem A.} {\it Let $ 1\le p<\infty$, $0<\alpha\leq 1$ and $w \in A_{p}$. Then the operators $G_{\alpha}$ and $\mathrm{g}_{\lambda,\alpha}^{*}$
are bounded from $L_{w}^{p}(l_2)$ into itself for $p>1$ and from $L_{w}^{1}(l_2)$ to $WL_{w}^{1}(l_2)$.}

Moreover, in \cite{Lerner}, Lerner showed sharp $L^{p}_{w}$ norm inequalities for the intrinsic square functions in terms of the $A_p$ characteristic constant of $w$ for all $1 < p < \i$.
Also Huang and Liu \cite{HuangLiu} studied the boundedness of intrinsic square functions on weighted Hardy spaces. Moreover, they characterized the weighted Hardy spaces by intrinsic square functions. In \cite{Wang2} and \cite{WangLiu}, Wang and Liu obtained some weak type estimates on weighted Hardy spaces. In \cite{Wang1}, Wang considered intrinsic functions and the commutators generated with BMO functions on weighted Morrey spaces. Let $b$ be a locally integrable function on $\Rn$. Setting
$$
A_{\alpha,b}^k f(t,y)\equiv\sup_{\phi\in C_{\alpha}}\left|\int_{\Rn}[b(x)-b(z)]^k\phi_{t}(y-z)f(z)dz\right|,
$$
the $k$th-order commutators are defined by
$$
[b,G_{\alpha}]^k f(x)=\left(\int\int_{\Gamma(x)}(A_{\alpha,b}^kf(t,y))^{2}\frac{dydt}{t^{n+1}}\right)^{\frac{1}{2}},
$$
$$
[b,\mathrm{g}_{\alpha}]^k f(x)=\left(\int_{0}^{\infty}(A_{\alpha,b}^kf(t,y))^{2}\frac{dt}{t}\right)^{\frac{1}{2}}
$$
and
$$
[b,\mathrm{g}_{\lambda,\alpha}^{*}]^k f(x)=\left(\int\int_{\mathbb{R}_{+}^{n+1}}\left(\frac{t}{t+|x-y|}\right)^{\lambda n}(A_{\alpha,b}^kf(t,y))^{2}\frac{dydt}{t^{n+1}}\right)^{\frac{1}{2}}.
$$
A function $b\in L_{1}^{loc}(\Rn)$ is said to be in $BMO(\Rn)$ if
$$
\|b\|_{*}=\sup_{x\in \Rn,\,r>0}\frac{1}{|B(x,r)|}\int_{B(x,r)}|b(y)-b_{B(x,r)}|dy<\i,
$$
where $b_{B(x,r)}=\frac{1}{|B(x,r)|}\int_{B(x,r)}b(y)dy$.

By the similar argument as in \cite{DingLuY} and \cite{Wang1}, we can get

\textbf{Theorem B.} {\it Let $ 1< p<\infty$, $0<\alpha\leq 1$, $w \in A_{p}$ and $b \in BMO(\Rn)$. Then the $k$th-order commutator operators $[b,G_{\alpha}]^k$ and $[b,\mathrm{g}_{\lambda,\alpha}^{*}]^k$
are bounded from $L_{w}^{p}(l_2)$ into itself.}

In this paper, we will consider the boundedness of the operators $G_{\alpha}$, $\mathrm{g}_{\alpha}$, $\mathrm{g}_{\lambda,\alpha}^{*}$ and their $k$th-order commutators on vector-valued generalized weighted Morrey spaces. Let $\varphi(x,\ r)$ be a positive measurable function on $\Rn\times \mathbb{R}_{+}$ and $w$ be non-negative measurable function on $\Rn$. For any $\vec{f}\in L^{p,loc}_{w}(l_2)$ , we denote by $M^{p,\varphi}_{w}(l_2)$ the vector-valued generalized weighted Morrey spaces, if
$$
\| \vec{f} \|_{M^{p,\varphi}_{w}(l_2)}=\sup_{x\in \Rn, \, r>0}\varphi(x,\ r)^{-1} \, w(B(x,r))^{-\frac{1}{p}} \,
\|\| \vec{f}(\cdot) \|_{l_2} \|_{L_{w}^{p}(B(x,r))}<\i.
$$
When $w\equiv 1$, then $M^{p,\varphi}_{w}(l_2)$ coincide the vector-valued generalized Morrey spaces $M^{p,\varphi}(l_2)$. There are many papers discussed the conditions on $\varphi(x,r)$ to obtain the boundedness of operators on the generalized Morrey spaces. For example, in \cite{GulDoc} (see, also \cite{GulJIA}), by Guliyev the following condition was imposed on the pair $(\varphi_{1},\varphi_{2})$ :
\begin{equation}\label{wueq1}
\int_{r}^{\infty}\varphi_{1}(x,t)\frac{dt}{t}\leq C \varphi_{2}(x,r).
\end{equation}
where $C>0$ does not depend on $x$ and $r$. Under the above condition, they obtained the boundedness of Calder\'{o}n-Zygmund singular integral operators from $M^{p,\varphi_{1}}(\Rn)$ to $M^{p,\varphi_{2}}(\Rn)$. Also, in \cite{AkbGulMust} and \cite{GULAKShIEOT2012}, Guliyev et. introduced a weaker condition: If $1\leq p<\i$, there exits a constant $C>0$, such that, for any $x\in \Rn$ and $r>0$,
\begin{equation}\label{wueq22}
\int_{r}^{\infty}\frac{\es_{t<s<\infty}\varphi_{1}(x,s)s^{\frac{n}{p}}}{t^{\frac{n}{p}+1}}dt\leq C\,\varphi_{2}(x,r).
\end{equation}
If the pair $(\varphi_{1},\varphi_{2})$ satisfies condition \eqref{wueq1}, then $(\varphi_{1},\varphi_{2})$ satisfied condition \eqref{wueq22}. But the opposite is not true. We can see remark 4.7 in \cite{GULAKShIEOT2012} for details.

Recently, in \cite{GulOMTSA, GulEMJ2012} (see, also \cite{GulKarMustSer, KarGulSer, MustAJM}), Guliyev introduced a weighted condition: If $1\leq p<\i$, there exits a constant $C>0$, such that, for any $x\in \Rn$ and $t>0$,
\begin{equation}\label{wueq2}
\int_{r}^{\infty} \frac{\es_{t<s<\infty} \varphi_1(x,s)
w(B(x,s))^{\frac{1}{p}}}{w(B(x,t))^{\frac{1}{p}}} \, \frac{dt}{t} \le
 C \, \varphi_2(x,r),
\end{equation}

In this paper, we will obtain the boundedness of the vector-valued intrinsic function, the intrinsic Littlewood-Paley $g$ function, the intrinsic $g_{\lambda}^{*}$ function and their $k$th-order commutators on vector-valued generalized weighted Morrey spaces when $w \in A_{p}$ and the pair $(\varphi_{1},\varphi_{2})$ satisfies condition \eqref{wueq2} or the following inequalities,
\begin{equation}\label{wueq3}
\int_{r}^{\infty}\ln^k \Big(e+\frac{t}{r}\Big) \, \frac{\es_{t<s<\infty} \varphi_1(x,s)
w(B(x,s))^{\frac{1}{p}}}{w(B(x,t))^{\frac{1}{p}}} \, \frac{dt}{t} \leq C\,\varphi_{2}(x,r),
\end{equation}
where $C$ does not depend on $x$ and $r$.
Our main results in this paper are stated as follows.

\begin{thm}\label{wuteo1.1}
Let $1\le p<\infty$, $0<\alpha\leq 1$, $w \in A_{p}$ and $(\varphi_{1},\varphi_{2})$ satisfies condition \eqref{wueq2}.
Then the operator $G_{\alpha}$ is bounded from $M^{p,\varphi_{1}}_{w}(l_2)$ to $M^{p,\varphi_{2}}_{w}(l_2)$ for $p>1$ and from $M^{1,\varphi_{1}}_{w}(l_2)$ to $WM^{1,\varphi_{2}}_{w}(l_2)$.
\end{thm}
\begin{thm}\label{wuteo1.2}
Let $1\le p<\infty$, $0<\alpha\leq 1$, $w \in A_{p}$, $\lambda >3+\displaystyle \frac{\alpha}{n}$ and $(\varphi_{1},\varphi_{2})$ satisfies condition \eqref{wueq2}.
Then the operator $\mathrm{g}_{\lambda,\alpha}^{*}$ is bounded from $M^{p,\varphi_{1}}_{w}(l_2)$ to $M^{p,\varphi_{2}}_{w}(l_2)$ for $p>1$ and from $M^{1,\varphi_{1}}_{w}(l_2)$ to $WM^{1,\varphi_{2}}_{w}(l_2)$.
\end{thm}
\begin{thm}\label{wuteo1.3}
Let $1<p<\infty$, $0<\alpha\leq 1$, $w \in A_{p}$, $b\in BMO$ and $(\varphi_{1},\varphi_{2})$ satisfies condition \eqref{wueq3}. Then $[b,G_{\alpha}]^k$ is bounded from $M^{p,\varphi_{1}}_{w}(l_2)$ to $M^{p,\varphi_{2}}_{w}(l_2)$ .
\end{thm}
\begin{thm}\label{wuteo1.4}
Let $1<p<\infty$, $0<\alpha\leq 1$, $w \in A_{p}$, $b\in BMO$ and $(\varphi_{1},\varphi_{2})$ satisfies condition \eqref{wueq3}, then for $\lambda >3+\frac{\alpha}{n}$,
 $[b,\mathrm{g}_{\lambda,\alpha}^{*}]^k$ is bounded from $M^{p,\varphi_{1}}_{w}(l_2)$ to $M^{p,\varphi_{2}}_{w}(l_2)$.
\end{thm}
In \cite{Wilson1}, the author proved that the functions $G_{\alpha}f$ and $\mathrm{g}_{\alpha}f$ are pointwise comparable. Thus, as a consequence of Theorem \ref{wuteo1.1} and Theorem \ref{wuteo1.3}, we have the following results.
\begin{cor}\label{wucor1.5}
Let $1\le p<\infty$, $0<\alpha\leq 1$, $w \in A_{p}$ and $(\varphi_{1},\varphi_{2})$ satisfies condition \eqref{wueq2}, then $\mathrm{g}_{\alpha}$ is bounded from $M^{p,\varphi_{1}}_{w}(l_2)$ to $M^{p,\varphi_{2}}_{w}(l_2)$ for $p>1$ and from $M^{1,\varphi_{1}}_{w}(l_2)$ to $WM^{1,\varphi_{2}}_{w}(l_2)$.
\end{cor}
\begin{cor}\label{wucor1.6}
Let $1<p<\infty$, $0<\alpha\leq 1$, $w \in A_{p}$, $b\in BMO$ and $(\varphi_{1},\varphi_{2})$ satisfies condition \eqref{wueq3}, then $[b,\mathrm{g}_{\alpha}]^k$ is bounded from $M^{p,\varphi_{1}}_{w}(l_2)$ to $M^{p,\varphi_{2}}_{w}(l_2)$.
\end{cor}

\begin{rem}
Note that, in the scalar valued case the Theorems \ref{wuteo1.1} - \ref{wuteo1.4} and Corollaries \ref{wucor1.5} - \ref{wucor1.6} was proved in \cite{GulShuk2013} ($w\equiv 1$) and 
\cite{GulMathN}. 
Also, in the scalar valued case and $w \equiv A_p$ and $\varphi_1(x,r) = \varphi_2(x,r) \equiv w(B(x,r))^{\frac{\kappa-1}{p}}$, $0<\kappa<1$ Theorems \ref{wuteo1.1}-\ref{wuteo1.4} and Corollaries \ref{wucor1.5}-\ref{wucor1.6} was proved by Wang in \cite{Wang1, Wang1W}. How as, if $\varphi(x,r) \equiv w(B(x,r))^{\frac{\kappa-1}{p}}$, then the vector-valued generalized weighed Morrey space $M^{p,\varphi}_{w}(l_2)$ coincide the vector-valued weighed Morrey space $L^{p,\kappa}_{w}(l_2)$ and the pair  $(w(B(x,r))^{\frac{\kappa-1}{p}}$, $w(B(x,r))^{\frac{\kappa-1}{p}})$ satisfies the both conditions
\eqref{wueq2} and \eqref{wueq3}. Indeed, by Lemma \ref{Graf} there exists $C>0$ and $\delta>0$ such that for all $x \in \Rn$ and $t>r$:
$$
w(B(x,t)) \ge C \Big(\frac{t}{r}\Big)^{n\delta} \, w(B(x,r)).
$$
Then
\begin{align*}
\int_{r}^{\infty} \frac{\es_{t<s<\infty} w(B(x,s))^{\frac{\kappa}{p}}}{w(B(x,t))^{1/p}} \, \frac{dt}{t}
& \le \int_{r}^{\infty} \ln^k \Big(e+\frac{t}{r}\Big) \, \frac{\es_{t<s<\infty} w(B(x,s))^{\frac{\kappa}{p}}}{w(B(x,t))^{1/p}} \, \frac{dt}{t}
\\
& = \int_{r}^{\infty} \ln^k \Big(e+\frac{t}{r}\Big) \, w(B(x,t))^{\frac{\kappa-1}{p}} \, \frac{dt}{t}
\end{align*}
\begin{align*}
& \lesssim \int_{r}^{\infty} \ln^k \Big(e+\frac{t}{r}\Big) \, \Big(\Big(\frac{t}{r}\Big)^{n\delta} \,  w(B(x,r))\Big)^{\frac{\kappa-1}{p}} \, \frac{dt}{t}
\\
& = w(B(x,r))^{\frac{\kappa-1}{p}} \, \int_{r}^{\infty} \ln^k \Big(e+\frac{t}{r}\Big) \, \Big(\frac{t}{r}\Big)^{n\delta \frac{\kappa-1}{p}} \, \frac{dt}{t}
\\
& = w(B(x,r))^{\frac{\kappa-1}{p}} \, \, \int_{1}^{\infty} \ln^k \Big(e+\tau\Big) \, \tau^{n\delta \frac{\kappa-1}{p}} \, \frac{d\tau}{\tau}
\\
& \thickapprox w(B(x,r))^{\frac{\kappa-1}{p}}.
\end{align*}
\end{rem}

Throughout this paper, we use the notation $A \lesssim B$ to mean that there is a positive constant $C$ independent of all essential variables such that $A\leq CB$. Moreover, $C$ may be different from place to place.

\

\section{Vector-valued generalized weighted Morrey spaces }

The classical Morrey spaces $M^{p,\lambda}$ were originally introduced by Morrey in \cite{Morrey} to study the local behavior of solutions to second order elliptic partial
differential equations. For the properties and applications of classical Morrey spaces, we refer the readers to \cite{Gi, KJF}.

We denote by $M^{p,\lambda}(l_2) \equiv M^{p,\lambda}(\Rn, l_2)$ the vector-valued Morrey space, the space of all vector-valued functions $\vec{f}\in L^{p,\rm loc}(l_2)$ with finite quasinorm
$$
  \left\| \vec{f}\right\|_{M^{p,\lambda}(l_2)} = \sup_{x\in \Rn, \; r>0} r^{-\frac{\lambda}{p}} \|\vec{f}\|_{L^{p}(B(x,r), l_2)},
$$
where $1\le p < \infty$ and $0 \le \lambda \le n$.

Note that $M^{p,0}(l_2)=L^{p}(l_2)$ and $M^{p,n}(l_2)=L^{\infty}(l_2)$. If $\lambda<0$ or $\lambda>n$, then $M^{p,\lambda}(l_2)={\Theta}$,
where $\Theta$ is the set of all vector-valued functions equivalent to $0$ on $\Rn$.

We define the vector-valued generalized weighed Morrey spaces as follows.
\begin{defn}
Let $1 \le p < \infty$, $\varphi$ be a positive measurable vector-valued function on $\Rn \times (0,\infty)$ and $w$ be non-negative
measurable function on $\Rn$. We denote by $M^{p,\varphi}_{w}(l_2)$ the vector-valued generalized weighted Morrey space, the
space of all vector-valued functions $\vec{f} \in L^{p,\rm loc}_{w}(l_2)$ with finite norm
$$
\|\vec{f}\|_{M^{p,\varphi}_{w}(l_2)} = \sup\limits_{x\in\Rn, r>0} \varphi(x,r)^{-1} \, w(B(x,r))^{-\frac{1}{p}} \, \|f\|_{L^{p}_{w}(B(x,r),l_2)},
$$
where $L^{p}_{w}(B(x,r),l_2)$ denotes the vector-valued weighted $L^{p}$-space of measurable functions
$f$ for which
\begin{equation*}
\|\vec{f}\|_{L^{p}_{w}(B(x,r))} \equiv \|\vec{f} \chi_{_{B(x,r)}}\|_{L^{p}_{w}(\Rn)}
= \left(\int_{B(x,r)}\|\vec{f}(y)\|_{l_2}^p w(y)dy\right)^{\frac{1}{p}}.
\end{equation*}

Furthermore, by $WM^{p,\varphi}_{w}(l_2)$ we denote the vector-valued weak generalized weighted Morrey space of all functions $f\in WL^{p,\rm loc}_{w}(l_2)$ for which
$$
\|\vec{f}\|_{WM^{p,\varphi}_{w}(l_2)} = \sup\limits_{x\in\Rn, r>0}
\varphi(x,r)^{-1} \, w(B(x,r))^{-\frac{1}{p}} \, \|\vec{f}\|_{WL^{p}_{w}(B(x,r),l_2)} < \infty,
$$
where $WL^{p}_{w}(B(x,r),l_2)$ denotes the weak $L^{p}_{w}$-space of measurable functions
$f$ for which
\begin{equation*}
\|\vec{f}\|_{WL^{p}_{w}(B(x,r),l_2)} \equiv \|\vec{f} \chi_{_{B(x,r)}}\|_{WL^{p}_{w}(l_2)}
= \sup_{t>0} t \left(\int_{\{y\in B(x,r) : \, \|\vec{f}(y)\|_{l_2} > t \}} w(y)dy\right)^{\frac{1}{p}}.
\end{equation*}
\end{defn}

\begin{rem}
$(1)~$ If $w\equiv 1$, then $M^{p,\varphi}_{1}(l_2)=M^{p,\varphi}(l_2)$ is the vector-valued generalized Morrey space.

$(2)~$ If $\varphi(x,r) \equiv w(B(x,r))^{\frac{\kappa-1}{p}}$, then
$M^{p,\varphi}_{w}(l_2)=L^{p,\kappa}_{w}(l_2)$ is the vector-valued weighted Morrey space.

$(3)~$ If $\varphi(x,r) \equiv v(B(x,r))^{\frac{\kappa}{p}} w(B(x,r))^{-\frac{1}{p}}$, then
$M^{p,\varphi}_{w}(l_2)=L^{p,\kappa}_{v,w}(l_2)$ is the vector-valued two weighted Morrey space.

$(4)~$ If $w\equiv1$ and $\varphi(x,r)=r^{\frac{\lambda-n}{p}}$ with $0<\lambda<n$, then $M^{p,\varphi}_{w}(l_2)=L^{p,\lambda}(l_2)$ is the
vector-valued Morrey space and $WM^{p,\varphi}_{w}(l_2)=WL^{p,\lambda}(l_2)$ is the vector-valued weak Morrey space.

$(5)~$ If $\varphi(x,r) \equiv w(B(x,r))^{-\frac{1}{p}}$, then $M^{p,\varphi}_{w}(l_2)=L^{p}_{w}(l_2)$ is the vector-valued weighted Lebesgue space.
\end{rem}

\

\section{Preliminaries and some lemmas}

By a weight function, briefly weight, we mean a locally integrable function on $\Rn$ which takes values in
$(0,\infty)$ almost everywhere. For a weight $w$ and a measurable set $E$,
we define $w(E) = \int_{E} w(x) dx$, and denote the Lebesgue measure of $E$ by $|E|$ and the
characteristic function of $E$ by $\chi_{_{E}}$. Given a weight $w$, we say that
$w$ satisfies the doubling condition if there exists a constant $D > 0$ such that
for any ball $B$, we have $w(2B) \le D w(B)$. When $w$ satisfies this condition,
we write brevity $w \in \Delta_{2}$.

If  $w$ is a weight function, we denote by $L^{p}_{w}(l_2)\equiv L^{p}_{w}(\Rn,l_2)$ the
vector-valued weighted Lebesgue space defined by finiteness of the norm
$$
\|\vec{f}\|_{L^{p}_{w}(l_2)}=\left(\int_{\Rn} \|\vec{f}(x)\|_{l_2}^p w(x)dx\right)^{\frac{1}{p}}<\infty, ~~~  \mbox{if} ~~ 1\le p<\infty
$$
and by $\|\vec{f}\|_{L^{\i}_{w}(l_2)}=\ess\limits_{x\in\Rn} \|\vec{f}(x)\|_{l_2} w(x)$ if $p=\infty$.

We recall that a weight function $w$ is in the Muckenhoupt's class $A_{p}$ \cite{Muckenh}, $1<p<\infty$, if
\begin{align*}
[w]_{A_{p}}: & = \sup\limits_{B} [w]_{A_{p}(B)} 
\\
& = \sup\limits_{B}
\left(\frac{1}{|B|}\int_{B} w(x)dx\right)\left(\frac{1}{|B|}\int_{B}w(x)^{1-p^{\prime}}dx\right)^{p-1}<\infty,
\end{align*}
where the $\sup$ is taken with respect to all the balls $B$ and
$\frac{1}{p}+\frac{1}{p^{\prime}}=1$. Note that, for all balls $B$ by H\"{o}lder's inequality
\begin{equation*}  \label{ghj5}
[w]_{A_{p}(B)}^{1/p}=|B|^{-1} \|w \|_{L^1(B)}^{1/p} \, \|w^{-1/p} \|_{L^{p'}(B)} \ge 1.
\end{equation*}
For $p=1$, the class $A_1$ is defined
by the condition $ Mw(x)\leq Cw(x) $ with $[w]_{A_1}= \sup\limits_{x\in\Rn}
\frac{Mw(x)}{w(x)},$  and for $p=\infty$ $~A_{\infty}=\bigcup_{1\le p<\infty}A_{p}$ and
$[w]_{A_{\i}}=\inf\limits_{1 \le p < \i} [w]_{A_{p}}$.

\begin{lem} \label{Graf} {\rm(\cite{Grafakos})}
$(1)~$ If $w \in A_p$ for some $1 \le p < \infty$, then $w \in \Delta_{2}$. Moreover, for all $\lambda > 1$
$$
w(\lambda B) \le \lambda^{np}[w]_{A_p} w(B).
$$

$(2)~$ If $w \in A_{\infty}$, then $w \in \Delta_{2}$. Moreover, for all $\lambda > 1$
$$
w(\lambda B) \le 2^{\lambda^{n}} [w]_{A_{\infty}} w(B).
$$

$(3)~$ If $w \in A_p$ for some $1 \le p \le \infty$, then there exit
$C>0$ and $\delta > 0$ such that for any ball $B$ and a measurable
set $S \subset B$,
$$
\frac{w(S)}{w(B)} \le C \Big( \frac{|S|}{|B|} \Big)^{\delta}.
$$

\end{lem}

We are going to use the following result on the boundedness of the Hardy operator
$$
(Hg)(t):=\frac{1}{t}\int_0^t g(r)d\mu(r),~ 0<t<\infty,
$$
where $\mu$ is a non-negative Borel measure on $(0,\infty)$.
\begin{thm}\label{thm3.2.XX}{\rm(\cite{CarPickSorStep})}
The inequality
\begin{equation*}
\ess_{t>0}\omega(t) Hg(t) \leq c \ess_{t>0}v(t)g(t)
\end{equation*}
holds for all functions $g$ non-negative and non-increasing on $(0,\i)$ if and
only if
\begin{equation*}
A:= \sup_{t>0}\frac{\omega(t)}{t}\int_0^t
\frac{d\mu(r)}{\ess_{0<s<r}v(s)}<\i,
\end{equation*}
and $c\thickapprox A$.
\end{thm}

We also need the following statement on the boundedness of the Hardy type operator
$$
(H_1 g)(t):=\frac{1}{t}\int_0^t \ln^k \Big(e+\frac{t}{r}\Big) \,  g(r)d\mu(r),~ 0<t<\infty,
$$
where $\mu$ is a non-negative Borel measure on $(0,\infty)$.
\begin{thm}\label{thm3.2.X}
The inequality
\begin{equation*}
\ess_{t>0}\omega(t) H_1 g(t) \leq c \ess_{t>0}v(t)g(t)
\end{equation*}
holds for all functions $g$ non-negative and non-increasing on $(0,\i)$ if and
only if
\begin{equation*}
A_1:= \sup_{t>0}\frac{\omega(t)}{t}\int_0^t \ln^k \Big(e+\frac{t}{r}\Big) \, \frac{d\mu(r)}{\ess_{0<s<r}v(s)}<\i,
\end{equation*}
and $c\thickapprox A_1$.
\end{thm}
Note that, Theorem \ref{thm3.2.X} can be proved analogously to Theorem 4.3 in \cite{GulAlKar1}.

\begin{defn} \label{BMOnorm}
$BMO(\Rn)$ is the Banach space modulo constants with the norm $\|\cdot\|_\ast$ defined by
\begin{equation*}
\|b\|_{\ast}=\sup_{x\in\Rn, r>0}\frac{1}{|B(x,r)|}
\int_{B(x,r)}|b(y)-b_{B(x,r)}|dy<\infty,
\end{equation*}
where $b\in L_1^{\rm loc}(\Rn)$ and
$$
b_{B(x,r)}=\frac{1}{|B(x,r)|} \int_{B(x,r)} b(y)dy.
$$
\end{defn}

\begin{lem}\label{MuckWh} {\rm(\cite{MuckWh1}, Theorem 5, p. 236)}
Let $w \in A_{\infty}$. Then the norm $\|\cdot\|_{\ast}$ is equivalent to the norm
$$
\| b \|_{\ast,w} = \sup_{x\in\Rn, r>0}\frac{1}{w(B(x,r))}
\int_{B(x,r)}|b(y)-b_{B(x,r),w}| w(y)dy,
$$
where
$$
b_{B(x,r),w}=\frac{1}{w(B(x,r))} \int_{B(x,r)} b(y) w(y) dy.
$$
\end{lem}

\begin{rem} \label{lem2.4.}
$(1)~$ The John-Nirenberg inequality : there are constants $C_1$,
$C_2>0$, such that for all $b \in BMO(\Rn)$ and $\beta>0$
$$
\left| \left\{ x \in B \, : \, |b(x)-b_{B}|>\beta \right\}\right|
\le C_1 |B| e^{-C_2 \beta/\| b \|_{\ast}}, ~~~ \forall B \subset \Rn.
$$

$(2)~$ For $1\le p<\i$ the John-Nirenberg inequality implies that
\begin{equation} \label{lem2.4.I}
\|b\|_{\ast}  \thickapprox \sup_{B}\left(\frac{1}{|B|}
\int_{B}|b(y)-b_{B}|^p dy\right)^{\frac{1}{p}}
\end{equation}
and for $1\le p<\i$ and $w \in A_{\infty}$
\begin{equation} \label{lem2.4.II}
\|b\|_{\ast}  \thickapprox \sup_{B}\left(\frac{1}{w(B)}
\int_{B}|b(y)-b_{B}|^p w(y)dy\right)^{\frac{1}{p}}.
\end{equation}
\end{rem}

Note that, by the John-Nirenberg inequality and Lemma \ref{Graf} (part 3) it follows that
$$
w(\{ x \in B \, : \, |b(x)-b_{B}|> \beta \})
\le C_1^{\delta} w(B) e^{-C_2 \beta \delta/\| b \|_{\ast}}
$$
for some $\delta>0$. Hence
\begin{align*}
\int_{B}|b(y)-b_{B}|^p w(y) dy & = p \int_{0}^{\infty} \beta^{p-1} \;
w(\{ x \in B \, : \, |b(x)-b_{B}|> \beta \}) d\beta
\\
& \le p C_1^{\delta} \, w(B) \, \int_{0}^{\infty} \beta^{p-1} \; e^{-C_2 \beta \delta/\| b \|_{\ast}} \, d\beta
= C_3 w(B) \|b\|_{\ast}^p,
\end{align*}
where $C_3>0$ depends only on $C_1^{\delta}$, $C_2$, $p$, and $\delta$, which implies \eqref{lem2.4.II}.

Also \eqref{lem2.4.I} is a particular case of \eqref{lem2.4.II} with $w \equiv 1$.

The following lemma was proved in \cite{GulEMJ2012}.
\begin{lem}\label{LinLU} 
i) Let $w \in A_{\infty}$ and $b \in BMO(\Rn)$.
Let also $1 \le p < \infty$, $x \in \Rn$, $k > 0$ and $r_1, r_2 > 0$. Then
$$
\Big( \frac{1}{w(B(x,r_1))} \int\limits_{B(x,r_1)} |b(y)-b_{B(x,r_2),w}|^{kp} w(y)dy\Big)^{\frac{1}{p}}
\le C  \, \left( 1+ \Big|\ln\frac{r_1}{r_2} \Big| \right)^{k} \| b \|_{\ast}^{k},
$$
where $C>0$ is independent of $f$, $w$, $x$, $r_1$, and $r_2$.

ii) Let $w \in A_p$ and $b \in BMO(\Rn)$.
Let also $1 < p < \infty$, $x \in \Rn$, $k > 0$ and $r_1, r_2 > 0$. Then
$$
\Big( \frac{1}{w^{1-p'}(B(x,r_1))} \int\limits_{B(x,r_1)} |b(y)-b_{B(x,r_2),w}|^{kp'} w(y)^{1-p'} dy\Big)^{\frac{1}{p'}}
\le C  \, \left( 1+ \Big|\ln\frac{r_1}{r_2} \Big| \right)^{k} \| b \|_{\ast}^{k},
$$
where $C>0$ is independent of $f$, $w$, $x$, $r_1$, and $r_2$.
\end{lem}

\

\section{Proofs of main theorems}

Before proving the main theorems, we need the following lemmas.
\begin{lem}\label{wulem2.3} \cite{Wang1}
For $j\in \mathbb{Z}_{+}$, denote
$$
G_{\alpha,2^{j}}(f)(x)=\left(\int_{0}^{\infty}\int_{|x-y|\leq 2^{j}t}(A_{\alpha}f(y,t))^{2}\frac{dydt}{t^{n+1}}\right)^{\frac{1}{2}}
$$
Let $0<\alpha\leq 1$, $1<p<\i$ and $w \in A_{p}$. Then any $j \in \mathbb{Z}_{+}$, we have
$$
\| G_{\alpha,2^{j}}(f)\|_{L_{w}^{p}}\lesssim 2^{j\big(\frac{3n}{2}+\alpha}\big) \, \| G_{\alpha}(f)\|_{L_{w}^{p}}.
$$
\end{lem}
This lemma is easy from the following inequality which is proved in \cite{Wilson1}.
$$
G_{\alpha,\beta}(f)(x)\leq\beta^{\frac{3n}{2}+\alpha}G_{\alpha}(f)(x).
$$
By the similar argument as in [3], we can get the following lemma.
\begin{lem}\label{wulem2.4}
Let $ 1<p<\i$, $0<\alpha\leq 1$ and $w \in A_{p}$, then the commutators $[b,G_{\alpha}]^k$ is bounded from $L_{w}^{p}(l_2)$ to itself whenever $b\in BMO$.
\end{lem}
Now we are in a position to prove theorems.

\begin{lem}\label{lem3.3.WeightSq}
Let $1\le p<\infty$, $0<\alpha\leq 1$ and $w \in A_{p}$.

Then, for $p>1$  the inequality
\begin{equation*}\label{eq3.5.weightPot}
\|G_{\alpha} \vec{f}\|_{L_{w}^{p}(B,l_2)} \lesssim \big(w(B)\big)^{\frac{1}{p}} \int_{2r}^{\i}\|\vec{f}\|_{L_{w}^{p}\big(B(x_0,t),l_2\big)} \, \big(w(B(x_0,t))\big)^{-\frac{1}{p}} \, \frac{dt}{t}
\end{equation*}
holds for any ball $B=B(x_0,r)$ and for all $\vec{f}\in\Lplocl2$.

Moreover, for $p=1$ the inequality
\begin{equation*}\label{eq3.5.WSqweight}
\|G_{\alpha} \vec{f}\|_{WL_{w}^{1}(B,l_2)} \lesssim w(B) \int_{2r}^{\i}\|\vec{f}\|_{L_{w}^{1}\big(B(x_0,t),l_2\big)} \, \big(w(B(x_0,t))\big)^{-1} \, \frac{dt}{t},
\end{equation*}
holds for any ball $B=B(x_0,r)$ and for all $\vec{f}\in\L1locl2$.
\end{lem}
\begin{proof}
The main ideas of these proofs come from \cite{GulEMJ2012}.
For arbitrary $x \in\Rn$, set $B=B(x_0,r)$, $2B\equiv B(x_0,2r)$. We decompose $\vec{f}=\vec{f}_{0}+\vec{f}_{\i}$, where $\vec{f}_{0}(y)=\vec{f}(y)\chi_{2B}(y)$, $\vec{f}_{\i}(y)=\vec{f}(y)-\vec{f}_{0}(y)$.
Then,
$$
\| G_{\alpha}\vec{f}\|_{L_{w}^{p}\big(B(x_0,r),l_2\big)}\leq\| G_{\alpha}\vec{f}_{0}\|_{L_{w}^{p}\big(B(x_0,r),l_2\big)}+\| G_{\alpha}\vec{f}_{\i}\|_{L^{p}\big(B(x_0,r),l_2\big)}:=I+II.
$$

First, let us estimate I. By Theorem $\mathrm{A}$, we can obtain that
\begin{equation}\label{wueq5X}
I\displaystyle \leq \| G_{\alpha}\vec{f}_{0}\|_{L_{w}^{p}(l_2)} \lesssim \| \vec{f}_{0}\|_{L_{w}^{p}(l_2)} = \| \vec{f}\|_{L_{w}^{p}(2B,l_2)}.
\end{equation}

On the other hand,
\begin{align}  \label{sal01}
\| \vec{f}\|_{L_{w}^{p}(2B,l_2)} & \thickapprox |B| \| \vec{f}\|_{L_{w}^{p}(2B,l_2)} \int_{2r}^{\i}\frac{dt}{t^{n+1}} \notag
\\
& \le |B| \int_{2r}^{\i} \| \vec{f}\|_{L_{w}^{p}\big(B(x_0,t),l_2\big)} \frac{dt}{t^{n+1}}
\\
& \lesssim w(B)^{\frac{1}{p}} \|w^{-1/p}\|_{L_{p'}(B)} \, \int_{2r}^{\i} \| \vec{f}\|_{L_{w}^{p}\big(B(x_0,t),l_2\big)} \, \frac{dt}{t^{n+1}}    \notag
\\
& \lesssim w(B)^{\frac{1}{p}} \int_{2r}^{\i} \| \vec{f}\|_{L_{w}^{p}\big(B(x_0,t),l_2\big)} \, \|w^{-1/p}\|_{L_{p'}(B(x_0,t))} \, \frac{dt}{t^{n+1}} \notag
\\
& \lesssim  \, w(B)^{\frac{1}{p}} \, \int_{2r}^{\i} \| \vec{f}\|_{L_{w}^{p}\big(B(x_0,t),l_2\big)} \, \big(w(B(x_0,t))\big)^{-\frac{1}{p}} \, \frac{dt}{t}. \notag
\end{align}

Therefore from \eqref{wueq5X} and \eqref{sal01} we get
\begin{equation}\label{wueq5}
I\displaystyle \lesssim  \, w(B)^{\frac{1}{p}} \, \int_{2r}^{\i} \| \vec{f}\|_{L_{w}^{p}\big(B(x_0,t),l_2\big)} \, \big(w(B(x_0,t))\big)^{-\frac{1}{p}} \, \frac{dt}{t}.
\end{equation}

Then let us estimate II.
$$
\|\vec{f}*\displaystyle \phi_{t}(y)\|_{l_2} = \left\|t^{-n}\int_{|y-z|\leq t}\phi(\frac{y-z}{t})\vec{f}_{\i}(z)dz\right\|_{l_2}
\leq t^{-n} \int_{|y-z|\leq t} \|\vec{f}_{\i}(z) \|_{l_2} dz.
$$
Since $x\in B(x_0,r)$, $(y,t)\in\Gamma(x)$, we have $|z-x|\leq|z-y|+|y-x|\leq 2t$, and
$$
r\leq|z-x_0|-|x_0-x|\leq|x-z|\leq|x-y|+|y-z|\leq 2t.
$$
So, we obtain
\begin{align*}
\big\|G_{\alpha} \vec{f}_{\i}(x)\big\|_{l_2} &\leq \left(\int \int_{\Gamma(x)}\left(t^{-n}\int_{|y-z|\leq t} \|\vec{f}_{\i}(z)\|_{l_2} dz\right)^{2}\frac{dydt}{t^{n+1}}\right)^{\frac{1}{2}}
\\
{}&\leq \left(\int_{t>r/2}\int_{|x-y|<t}\left(\int_{|x-z|\leq 2t} \|\vec{f}_{\i}(z)\|_{l_2} dz\right)^{2}\frac{dydt}{t^{3n+1}}\right)^{\frac{1}{2}}
\\
{}&\lesssim \left(\int_{t>r/2}\left(\int_{|z-x|\leq 2t} \|\vec{f}_{\i}(z)\|_{l_2} dz\right)^{2}\frac{dt}{t^{2n+1}}\right)^{\frac{1}{2}}.
\end{align*}
By Minkowski and H\"{o}lder's inequalities and $|z-x|\displaystyle \geq|z-x_0|-|x_0-x|\geq\frac{1}{2}|z-x_0|$, we have
\begin{align} \label{pekin1}
\big\|G_{\alpha} \vec{f}_{\i}(x)\big\|_{l_2} & \lesssim \int_{\Rn}\left(\int_{t>\frac{|z-x|}{2}}\frac{dt}{t^{2n+1}}\right)^{\frac{1}{2}} \|\vec{f}_{\i}(z)\|_{l_2} dz  \notag
\\
{}&\lesssim\int_{|z-x_0|>2r}\frac{\|\vec{f}(z)\|_{l_2} }{|z-x|^{n}}dz\lesssim\int_{|z-x_0|>2r}\frac{\|\vec{f}(z)\|_{l_2} }{|z-x_0|^{n}}dz   \notag
\\
{}&=\int_{|z-x_0|>2r} \|\vec{f}(z)\|_{l_2} \int_{|z-x_0|}^{+\infty}\frac{dt}{t^{n+1}}dz     \notag
\\
{}&=\int_{2r}^{+\infty}\int_{2r<|z-x_0|<t} \|\vec{f}(z)\|_{l_2} dz \frac{dt}{t^{n+1}} \notag
\\
& \lesssim \int_{2r}^{\infty} \| \|\vec{f}(z)\|_{l_2} \|_{L_{w}^{p}(B(x_0,t))} \, \|w^{-1}\|_{L_{p'}(B(x_0,t))} \, \frac{dt}{t^{n+1}}  \notag
\\
& \lesssim \int_{2r}^{\infty} \| \vec{f} \|_{L_{w}^{p}\big(B(x_0,t),l_2\big)} \, \big(w(B(x_0,t))\big)^{-\frac{1}{p}} \, \frac{dt}{t}.
\end{align}
Thus,
\begin{equation}\label{wueq6}
\| G_{\alpha}\vec{f}_{\i}\|_{L_{w}^{p}(B,l_2)}\lesssim w(B)^{\frac{1}{p}} \, \int_{2r}^{\infty} \| \vec{f} \|_{L_{w}^{p}\big(B(x_0,t),l_2\big)} \, \big(w(B(x_0,t))\big)^{-\frac{1}{p}} \, \frac{dt}{t}.
\end{equation}
By combining \eqref{wueq5} and \eqref{wueq6}, we have
$$
\| G_{\alpha}\vec{f}\|_{L_{w}^{p}(B,l_2)} \lesssim w(B)^{\frac{1}{p}} \, \int_{2r}^{\infty} \| \vec{f} \|_{L_{w}^{p}\big(B(x_0,t),l_2\big)} \, \big(w(B(x_0,t))\big)^{-\frac{1}{p}} \, \frac{dt}{t}.
$$
\end{proof}

\textbf{Proof of Theorem \ref{wuteo1.1}}

By Lemma \ref{lem3.3.WeightSq} and Theorem \ref{thm3.2.XX} we have for $p>1$
\begin{align*}
\| G_{\alpha}\vec{f}\|_{M^{p,\varphi_2}_{w}(l_2)} & \lesssim\sup\limits_{x_0\in \Rn,r>0}\varphi_{2}(x_0,\ r)^{-1}\, \int_{r}^{\infty}
\| \vec{f} \|_{L_{w}^{p}\big(B(x_0,t),l_2\big)} \, \big(w(B(x_0,t))\big)^{-\frac{1}{p}} \, \frac{dt}{t}
\\
&=\sup\limits_{x_0\in \Rn,r>0}\varphi_{2}(x_0,\ r)^{-1} \int_{0}^{r^{-1}} \| \vec{f} \|_{L_{w}^{p}\big(B(x_0,t^{-1}),l_2\big)} \, \big(w(B(x_0,t^{-1}))\big)^{-\frac{1}{p}} \, \frac{dt}{t}
\\
&=\sup\limits_{x_0\in \Rn,r>0}\varphi_{2}(x_0,r^{-1})^{-1} \, r \, \frac{1}{r} \, \int_{0}^{r} \| \vec{f} \|_{L_{w}^{p}\big(B(x_0,t^{-1}),l_2\big)}
\, \big(w(B(x_0,t^{-1}))\big)^{-\frac{1}{p}} \, \frac{dt}{t}
\\
& \lesssim \sup\limits_{x_0\in \Rn,r>0} \varphi_{1}(x_0,r^{-1})^{-1} \, \big(w(B(x_0,r^{-1}))\big)^{-\frac{1}{p}} \, \| \vec{f} \|_{L_{w}^{p}\big(B(x_0,r^{-1}),l_2\big)}
\\
& = \sup\limits_{x_0\in \Rn, r>0} \varphi_{1}(x_0,r)^{-1} \, \big(w(B(x_0,r))\big)^{-\frac{1}{p}} \, \| \vec{f} \|_{L_{w}^{p}\big(B(x_0,r),l_2\big)} = \|\vec{f}\|_{M^{p,\varphi_1}_{w}(l_2)}
\end{align*}
and for $p=1$
\begin{align*}
\| G_{\alpha}\vec{f}\|_{WM^{1,\varphi_2}_{w}(l_2)} & \lesssim\sup\limits_{x_0\in \Rn,r>0} \varphi_{2}(x_0,r)^{-1}\, \int_{r}^{\infty}
\| \vec{f} \|_{L_{w}^{1}\big(B(x_0,t),l_2\big)} \, \big(w(B(x_0,t))\big)^{-1} \, \frac{dt}{t}
\\
&=\sup\limits_{x_0\in \Rn,r>0}\varphi_{2}(x_0,\ r)^{-1} \int_{0}^{r^{-1}} \| \vec{f} \|_{L_{w}^{1}\big(B(x_0,t^{-1}),l_2\big)} \, \big(w(B(x_0,t^{-1}))\big)^{-1} \, \frac{dt}{t}
\\
&=\sup\limits_{x_0\in \Rn,r>0}\varphi_{2}(x_0,r^{-1})^{-1} \, r \, \frac{1}{r} \, \int_{0}^{r} \| \vec{f} \|_{L_{w}^{1}\big(B(x_0,t^{-1}),l_2\big)}
\, \big(w(B(x_0,t^{-1}))\big)^{-1} \, \frac{dt}{t}
\\
& \lesssim \sup\limits_{x_0\in \Rn,r>0} \varphi_{1}(x_0,r^{-1})^{-1} \, \big(w(B(x_0,r^{-1}))\big)^{-1} \, \| \vec{f} \|_{L_{w}^{1}\big(B(x_0,r^{-1}),l_2\big)}
\\
& = \sup\limits_{x_0\in \Rn, r>0} \varphi_{1}(x_0,r)^{-1} \, \big(w(B(x_0,r))\big)^{-1} \, \| \vec{f} \|_{L_{w}^{1}\big(B(x_0,r),l_2\big)} = \|\vec{f}\|_{M^{1,\varphi_1}_{w}(l_2)}.
\end{align*}

\

\begin{lem}\label{lem3.3.Weightga}
Let $1\le p<\infty$, $0<\alpha\leq 1$, $\lambda >3+\displaystyle \frac{\alpha}{n}$ and $w \in A_{p}$.
Then, for $p>1$  the inequality
\begin{equation*} \label{pekin02}
\big\|\mathrm{g}_{\lambda,\alpha}^{*}(\vec{f})\big\|_{L_{w}^{p}\big(B,l_2\big)} \lesssim \big(w(B)\big)^{\frac{1}{p}} \int_{2r}^{\i} \| \vec{f} \|_{L_{w}^{p}\big(B(x_0,t),l_2\big)} \, \big(w(B(x_0,t))\big)^{-\frac{1}{p}} \, \frac{dt}{t}
\end{equation*}
holds for any ball $B=B(x_0,r)$ and for all $\vec{f}\in\Lplocl2$.

Moreover, for $p=1$ the inequality
\begin{equation*}\label{pekin02W}
\big\|\mathrm{g}_{\lambda,\alpha}^{*}(\vec{f})\big\|_{WL_{w}^{1}\big(B,l_2\big)} \lesssim w(B) \int_{2r}^{\i} \| \vec{f} \|_{L_{w}^{1}\big(B(x_0,t),l_2\big)} \, \big(w(B(x_0,t))\big)^{-1} \, \frac{dt}{t}
\end{equation*}
holds for any ball $B=B(x_0,r)$ and for all $\vec{f}\in\L1locl2$.
\end{lem}
\begin{proof}
From the definition of $\mathrm{g}_{\lambda,\alpha}^{*}(f)$, we readily see that
\begin{align*}
\big\|\mathrm{g}_{\lambda,\alpha}^{*}(\vec{f})(x)\big\|_{l_2} & = \Big\|\Big(\int_{0}^{\infty}\int_{\Rn}\left(\frac{t}{t+|x-y|}\right)^{n\lambda} \Big(A_{\alpha}\vec{f}(y,t)\Big)^2\frac{dydt}{t^{n+1}}\Big)^{l/2}\Big\|_{l_2}
\\
& \le \Big\| \Big(\int_{0}^{\infty}\int_{|x-y|<t}\left(\frac{t}{t+|x-y|}\right)^{n\lambda} \Big(A_{\alpha}\vec{f}(y,t)\Big)^2\frac{dydt}{t^{n+1}}\Big)^{l/2}\Big\|_{l_2}
\\
& + \Big\|  \Big(\int_{0}^{\infty}\int_{|x-y|\geq t}\left(\frac{t}{t+|x-y|}\right)^{n\lambda} \Big(A_{\alpha}\vec{f}(y,t)\Big)^2\frac{dydt}{t^{n+1}}\Big)^{l/2}\Big\|_{l_2}
\\
&: = III+IV.
\end{align*}
First, let us estimate III.
$$
III\leq \Big\| \Big(\int_{0}^{\infty}\int_{|x-y|<t}\left(\frac{t}{t+|x-y|}\right)^{n\lambda} \Big(A_{\alpha}\vec{f}(y,t)\Big)^2\frac{dydt}{t^{n+1}}\Big)^{l/2}\Big\|_{l_2} \leq \big\|G_{\alpha}\vec{f}(x) \big\|_{l_2}.
$$
Now, let us estimate IV.
\begin{align*}
IV &\leq \Big\| \Big( \sum_{j=1}^{\infty}\int_{0}^{\infty}\int_{2^{j-1}t\leq|x-y|\leq 2^{j}t}\left(\frac{t}{t+|x-y|}\right)^{n\lambda} \Big(A_{\alpha}\vec{f}(y,t)\Big)^2\frac{dydt}{t^{n+1}}\Big)^{l/2}\Big\|_{l_2}
\\
{}& \lesssim \Big\| \Big( \sum_{j=1}^{\infty}\int_{0}^{\infty}\int_{2^{j-1}t\leq|x-y|\leq 2^{j}t} 2^{-jn\lambda}
\, \Big(A_{\alpha}\vec{f}(y,t)\Big)^2\frac{dydt}{t^{n+1}}\Big)^{l/2}\Big\|_{l_2}
\\
{}&\lesssim \sum_{j=1}^{\infty}2^{-jn\lambda} \, \Big\| \Big( \int_{0}^{\infty}\int_{|x-y|\leq 2^{j}t} \Big(A_{\alpha}\vec{f}(y,t)\Big)^2\frac{dydt}{t^{n+1}}\Big)^{l/2}\Big\|_{l_2}
\\
{}&:=\sum_{j=1}^{\infty}2^{-jn\lambda} \Big\|G_{\alpha,2^{j}}(\vec{f})(x)\Big\|_{l_2}.
\end{align*}
Thus,
\begin{equation}\label{wueq7}
\| \mathrm{g}_{\lambda,\alpha}^{*}(\vec{f}) \|_{L_{w}^{p}\big(B,l_2\big)}\leq\| G_{\alpha}\vec{f} \displaystyle \|_{L_{w}^{p}\big(B,l_2\big)}+\sum_{j=1}^{\i} 2^{-\frac{jn\lambda}{2}}\| G_{\alpha,2^{j}}(\vec{f}) \|_{L_{w}^{p}\big(B,l_2\big)}.
\end{equation}
By Lemma \ref{lem3.3.WeightSq}, we have
\begin{equation}\label{wueq8FD}
\| G_{\alpha}\vec{f}\|_{L_{w}^{p}\big(B,l_2\big)}\lesssim   \big(w(B)\big)^{\frac{1}{p}} \int_{2r}^{\i}\|\vec{f}\|_{L_{w}^{p}(B(x_0,t))} \, \big(w(B(x_0,t))\big)^{-\frac{1}{p}} \, \frac{dt}{t}.
\end{equation}
In the following, we will estimate $\| G_{\alpha,2^{j}}(\vec{f}) \|_{L_{w}^{p}\big(B,l_2\big)}$.
We divide $\| G_{\alpha,2^{j}}(\vec{f})\|_{L_{w}^{p}\big(B,l_2\big)}$ into two parts.
\begin{equation}\label{wueq9}
\| G_{\alpha,2^{j}}(\vec{f})\|_{L_{w}^{p}\big(B,l_2\big)}\leq\| G_{\alpha,2^{j}}(\vec{f}_{0})\|_{L_{w}^{p}\big(B,l_2\big)}+\| G_{\alpha,2^{j}}(\vec{f}_{\i})\|_{L_{w}^{p}\big(B,l_2\big)},
\end{equation}
where $\vec{f}_{0}(y)=\vec{f}(y)\chi_{2B}(y)$, $\vec{f}_{\i}(y)=\vec{f}(y)-\vec{f}_{\i}(y)$. For the first part, by Lemma \ref{wulem2.3},
\begin{align}\label{wueq10}
\| G_{\alpha,2^{j}}(\vec{f}_{0})\|_{L_{w}^{p}\big(B,l_2\big)} & \lesssim 2^{j(\frac{3n}{2}+\alpha)}\| G_{\alpha}(\vec{f}_{0})\|_{L_{w}^{p}(l_2)}\lesssim 2^{j(\frac{3n}{2}+\alpha)}\| f\|_{L_{w}^{p}\big(B,l_2\big)} \notag
\\
{}&\lesssim  2^{j(\frac{3n}{2}+\alpha)}w(B)^{\frac{1}{p}} \, \int_{2r}^{\infty}\| \vec{f}\|_{L_{w}^{p}\big(B(x_0,t),l_2)} \, \big(w(B(x_0,t))\big)^{-\frac{1}{p}} \, \frac{dt}{t}.
\end{align}
For the second part.
\begin{align*}
\Big\|G_{\alpha,2^{j}}(\vec{f}_{\i})(x)\Big\|_{l_2} & = \Big\|  \Big(\int_{0}^{\infty} \int_{|x-y|\leq 2^{j}t} \Big(A_{\alpha}\vec{f}(y,t)\Big)^2\frac{dydt}{t^{n+1}}\Big)^{l/2}\Big\|_{l_2}
\\
& =\Big\| \left(\int_{0}^{\infty}\int_{|x-y|\leq 2^{j}t}\left(\sup_{\phi\in C_{\alpha}}|\vec{f}*\phi_{t}(y)|\right)^{2}\frac{dydt}{t^{n+1}}\right)^{\frac{1}{2}} \Big\|_{l_2}
\\
& \leq \left(\int_{0}^{\infty}\int_{|x-y|\leq 2^{j}t}\left(\int_{|z-y|\leq t}\|\vec{f}_{\i}(z)\|_{l_2}dz\right)^{2}\frac{dydt}{t^{3n+1}}\right)^{\frac{1}{2}}.
\end{align*}
Since $|x-z|\leq|y-z|+|x-y|\leq 2^{j+1}t$, we get
\begin{align*}
\Big\|G_{\alpha,2^{j}}(\vec{f}_{\i})(x)\Big\|_{l_2} &\leq \left(\displaystyle \int_{0}^{\infty}\int_{|x-y|\leq 2^{j}t}\left(\int_{|x-z|\leq 2^{j+1}t}\|\vec{f}_{\i}(z)\|_{l_2}dz\right)^{2}\frac{dydt}{t^{3n+1}}\right)^{\frac{1}{2}}
\\
&\leq \left(\int_{0}^{\infty}\left(\int_{|z-x|\leq 2^{j+1}t} \|\vec{f}_{\i}(z)\|_{l_2} dz\right)^{2}\frac{2^{jn}dt}{t^{2n+1}}\right)^{\frac{1}{2}}
\\
&\leq 2^{\frac{jn}{2}}\int_{\Rn}\left(\int_{t\geq\frac{|x-z|}{2^{j+1}}} \|\vec{f}_{\i}(z)\|_{l_2}^{2}\frac{dt}{t^{2n+1}}\right)^{\frac{1}{2}}dz
\\
&\leq 2^{\frac{3jn}{2}}\int_{|x_0-z|>2r}\frac{\|\vec{f}(z)\|_{l_2}}{|x-z|^{n}}dz.
\end{align*}
For $|z-x|\displaystyle \geq|x_0-z|-|x-x_0|\geq|x_0-z|-\frac{1}{2}|x_0-z|=\frac{1}{2}|x_0-z|$, so by Fubini's theorem and H\"{o}lder's inequality, we obtain
\begin{align*}
\Big\|G_{\alpha,2^{j}}(\vec{f}_{\i})(x)\Big\|_{l_2} & \leq 2^{\frac{3jn}{2}}\int_{|x_0-z|>2r}\frac{\|\vec{f}(z)\|_{l_2}}{|x_0-z|^{n}}dz
\\
{}&= 2^{\frac{3jn}{2}}\int_{|x_0-z|>2r}\|\vec{f}(z)\|_{l_2}\int_{|x_0-z|}^{\infty}\frac{dt}{t^{n+1}}dz
\\
{}&\leq 2^{\frac{3jn}{2}}\int_{2r}^{\infty}\int_{|x_0-z|<t} \|\vec{f}(z)\|_{l_2} dz \frac{dt}{t^{n+1}}
\\
{}&\leq 2^{\frac{3jn}{2}}\int_{2r}^{\infty}\| \|\vec{f}(\cdot)\|_{l_2}\|_{L^{1}(B(x_0,t))}\frac{dt}{t^{n+1}}.
\\
& \leq 2^{\frac{3jn}{2}} \int_{2r}^{\infty} \| \vec{f}(\cdot)\|_{l_2} \|_{L_{w}^{p}(B(x_0,t))} \, \|w^{-1}\|_{L_{p'}(B(x_0,t))} \, \frac{dt}{t^{n+1}}
\\
& \leq 2^{\frac{3jn}{2}} \int_{2r}^{\infty} \| \vec{f}\|_{L_{w}^{p}\big(B(x_0,t),l_2\big)} \, \big(w(B(x_0,t))\big)^{-\frac{1}{p}} \, \frac{dt}{t}.
\end{align*}
So,
\begin{equation}\label{wueq11}
\| G_{\alpha,2^{j}}(\vec{f}_{\i})\|_{L_{w}^{p}\big(B,l_2\big)} \leq 2^{\frac{3jn}{2}}w(B)^{\frac{1}{p}} \, \int_{2r}^{\infty} \| \vec{f}\|_{L_{w}^{p}\big(B(x_0,t),l_2\big)} \, \big(w(B(x_0,t))\big)^{-\frac{1}{p}} \, \frac{dt}{t}.
\end{equation}
Combining \eqref{wueq9}, \eqref{wueq10} and \eqref{wueq11}, we have
\begin{equation}\label{wueq11s1}
\| G_{\alpha,2^{j}}(\vec{f})\|_{L_{w}^{p}\big(B,l_2\big)}\lesssim 2^{j(\frac{3n}{2}+\alpha)} \, w(B)^{\frac{1}{p}} \,
\int_{2r}^{\infty} \| \vec{f}\|_{L_{w}^{p}\big(B(x_0,t),l_2\big)} \, \big(w(B(x_0,t))\big)^{-\frac{1}{p}} \, \frac{dt}{t}.
\end{equation}

Thus,
\begin{equation}\label{wueq7B}
\| \mathrm{g}_{\lambda,\alpha}^{*}(\vec{f}) \|_{L_{w}^{p}\big(B,l_2\big)}\leq\| G_{\alpha}\vec{f} \|_{L_{w}^{p}\big(B,l_2\big)}+\sum_{j=1}^{\i} 2^{-\frac{jn\lambda}{2}}\| G_{\alpha,2^{j}}(\vec{f}) \|_{L_{w}^{p}\big(B,l_2\big)}.
\end{equation}

Since $\lambda >3+\displaystyle \frac{\alpha}{n}$, by \eqref{wueq8FD}, \eqref{wueq11s1} and \eqref{wueq7B}, we have the desired lemma.
\end{proof}

\

\textbf{Proof of Theorem \ref{wuteo1.2}}

From inequality \eqref{wueq7} we have
\begin{equation}\label{wueq7}
\| \mathrm{g}_{\lambda,\alpha}^{*}(\vec{f}) \|_{M^{p,\varphi_2}_{w}(l_2)}\leq\| G_{\alpha}\vec{f} \displaystyle \|_{M^{p,\varphi_2}_{w}(l_2)}+\sum_{j=1}^{\i} 2^{-\frac{jn\lambda}{2}}\| G_{\alpha,2^{j}}(\vec{f}) \|_{M^{p,\varphi_2}_{w}(l_2)}.
\end{equation}
By Theorem \ref{wuteo1.1}, we have
\begin{equation}\label{wueq8}
\| G_{\alpha}\vec{f}\|_{M^{p,\varphi_2}_{w}(l_2)}\lesssim\| \vec{f}\|_{M^{p,\varphi_1}_{w}(l_2)}.
\end{equation}
In the following, we will estimate $\| G_{\alpha,2^{j}}(\vec{f}) \|_{M^{p,\varphi_2}_{w}(l_2)}$.
Thus, by substitution of variables and Theorem \ref{thm3.2.XX}, we get
\begin{align}\label{wueq12}
&\| G_{\alpha,2^{j}}(\vec{f})\|_{M^{p,\varphi_2}_{w}(l_2)} \notag
\\
& \lesssim  2^{j(\frac{3n}{2}+\alpha)}\sup_{x_0\in \Rn,r>0}\varphi_{2}(x_0,r)^{-1} \int_{r}^{\infty}
\| \vec{f}\|_{L_{w}^{p}\big(B(x_0,t),l_2\big)} \, \big(w(B(x_0,t))\big)^{-\frac{1}{p}} \, \frac{dt}{t} \notag
\\
&=2^{j(\frac{3n}{2}+\alpha)}\displaystyle \sup_{x_0\in \Rn,r>0} \varphi_{2}(x_0,r^{-1})^{-1} \, r \, \frac{1}{r} \, \int_{0}^{r} \| \vec{f}\|_{L_{w}^{p}\big(B(x_0,t^{-1}),l_2\big)}
\, \big(w(B(x_0,t^{-1}))\big)^{-\frac{1}{p}} \, \frac{dt}{t}  \notag
\\
&\lesssim 2^{j(\frac{3n}{2}+\alpha)}\sup_{x_0\in \Rn,r>0} \varphi_{1}(x_0,r^{-1})^{-1} \, \big(w(B(x_0,r^{-1}))\big)^{-\frac{1}{p}} \, \| \vec{f}\|_{L_{w}^{p}\big(B(x_0,r^{-1}),l_2\big)} \notag
\\
&=2^{j(\frac{3n}{2}+\alpha)}\| \vec{f}\|_{M^{p,\varphi_1}_{w}(l_2)}.
\end{align}
Since $\lambda >3+\displaystyle \frac{\alpha}{n}$, by \eqref{wueq7}, \eqref{wueq8} and \eqref{wueq12}, we have the desired theorem.

\

\begin{lem}\label{lem3.3.ComWeightSq}
Let $1< p<\infty$, $0<\alpha\leq 1$, $w \in A_{p}$ and $b\in BMO$.

Then the inequality
\begin{equation*}
\|[b,G_{\alpha}]^k \vec{f}\|_{L_{w}^{p}\big(B,l_2\big)} \lesssim
\big(w(B)\big)^{\frac{1}{p}} \int_{2r}^{\i} \ln^k \Big(e+\frac{t}{r}\Big) \,  \| \vec{f}\|_{L_{w}^{p}\big(B(x_0,t),l_2\big)} \, \big(w(B(x_0,t))\big)^{-\frac{1}{p}} \, \frac{dt}{t}
\end{equation*}
holds for any ball $B=B(x_0,r)$ and for all $f\in\Lplocl2$.
\end{lem}
\begin{proof} We decompose $\vec{f}=\vec{f}_{0}+\vec{f}_{\i}$, where $\vec{f}_{0}=\vec{f}\chi_{2B}$ and $\vec{f}_{\i}=\vec{f}-\vec{f}_{0}$. Then
$$
\|[b,G_{\alpha}]^k\vec{f}\|_{L_{w}^{p}\big(B,l_2\big)}\leq \|[b,G_{\alpha}]^k\vec{f}_{0}\|_{L_{w}^{p}\big(B,l_2\big)}+\|[b,G_{\alpha}]^k\vec{f}_{\i}\|_{L_{w}^{p}\big(B,l_2\big)}.
$$
By Lemma \ref{wulem2.4}, we have that
\begin{align*}
\|[b,\ G_{\alpha}]^k\vec{f}_{0}\|_{L_{w}^{p}\big(B,l_2\big)} & \lesssim \|b\|_{*}^k \, \| \vec{f}_{0}\|_{L_{w}^{p}(l_2)}= \|b\|_{*}^k \, \| \vec{f}\|_{L_{w}^{p}\big(2B,l_2\big)}
\\
& \lesssim \|b\|_{*}^k \, w(B)^{\frac{1}{p}} \, \int_{2r}^{\infty} \| \vec{f}\|_{L_{w}^{p}\big(B(x_0,t),l_2\big)} \, \, \big(w(B(x_0,t))\big)^{-\frac{1}{p}} \, \frac{dt}{t}.
\end{align*}
For the second part, we divide it into two parts.
\begin{align*}
\Big\|[b,G_{\alpha}]^k\vec{f}_{\i}(x)\Big\|_{l_2} &= \Big\|\Big(\displaystyle \int\int_{\Gamma(x)}\sup_{\phi\in C_{\alpha}}\Big|\int_{\Rn}[b(x)-b(z)]^k\phi_{t}(y-z)\vec{f}_{\i}(z)dz\Big|^{2}\frac{dydt}{t^{n+1}}\Big)^{\frac{1}{2}}\Big\|_{l_2}
\\
{}&\leq \Big\|\Big(\int\int_{\Gamma(x)}\sup_{\phi\in C_{\alpha}}\Big|\int_{\Rn}[b(x)-b_{B,w}]^k\phi_{t}(y-z)\vec{f}_{\i}(z)dz\Big|^{2}\frac{dydt}{t^{n+1}}\Big)^{\frac{1}{2}}\Big\|_{l_2}
\\
{}&{}+\Big\|\Big(\int\int_{\Gamma(x)}\sup_{\phi\in C_{\alpha}}\Big|\int_{\Rn}[b(z)-b_{B,w}]^k\phi_{t}(y-z)\vec{f}_{\i}(z)dz\Big|^{2}\frac{dydt}{t^{n+1}}\Big)^{\frac{1}{2}}\Big\|_{l_2}
\\
{}&:= A(x)+B(x).
\end{align*}
Therefore
$$
\|[b,G_{\alpha}]^k\vec{f}_{\i}\|_{L_{w}^{p}\big(B,l_2\big)} \le \|A(\cdot)\|_{L_{w}^{p}(B)} + \|B(\cdot)\|_{L_{w}^{p}(B)}.
$$

First, for $A(x)$, we find that
\begin{align*}
A(x) & = |b(x)-b_{B,w}|^k \, \Big\|\Big(\iint_{\Gamma(x)}\sup_{\phi \in C_{\alpha}} \Big|\int_{\Rn}\phi_{t}(y-z)\vec{f}_{\i}(z)dz\Big|^{2} \frac{dydt}{t^{n+1}}\Big)^{\frac{1}{2}}\Big\|_{l_2}
\\
& = \big|b(x)-b_{B,w}\big|^k \, \big\|G_{\alpha}\vec{f}_{\i}(x)\big\|_{l_2}.
\end{align*}
By Lemma \ref{LinLU} and from the inequality \eqref{pekin1}, we can get
\begin{align*}
\|A(\cdot)\|_{L_{w}^{p}(B)} & = \left(\int_{B}|b(x)-b_{B,w}|^{kp} \, \Big(\big\|G_{\alpha}\vec{f}_{\i}(x)\big\|_{l_2}\Big)^{p} \, w(x) dx\right)^{\frac{1}{p}}
\\
&\leq\left(\int_{B}|b(x)-b_{B,w}|^{kp} \, w(x) dx\right)^{\frac{1}{p}} \int_{2r}^{\infty} \| \vec{f}\|_{L_{w}^{p}\big(B(x_0,t),l_2\big)} \, \big(w(B(x_0,t))\big)^{-\frac{1}{p}} \, \frac{dt}{t}
\\
&\leq\| b\|_{*}^k w(B)^{\frac{1}{p}} \, \int_{2r}^{\infty} \| \vec{f}\|_{L_{w}^{p}\big(B(x_0,t),l_2\big)} \, \big(w(B(x_0,t))\big)^{-\frac{1}{p}} \, \frac{dt}{t}.
\end{align*}
For $B(x)$, since $|y-x|<t$, we get $|x-z|<2t$. Thus, by Minkowski's inequality,
\begin{align*}
B(x) &\leq \Big\|\Big( \int\int_{\Gamma(x)}\Big|\int_{|x-z|<2t}|b_{B,w}-b(z)|^k \, \vec{f}_{\i}(z)dz\Big|^{2}\frac{dydt}{t^{3n+1}}\Big)^{\frac{1}{2}}\Big\|_{l_2}
\\
&\lesssim \Big( \int_{0}^{\infty}\Big|\int_{|x-z|<2t}|b_{B,w}-b(z)|^k\, \big\|\vec{f}_{\i}(z)\big\|_{l_2} dz\Big|^{2}\frac{dt}{t^{2n+1}}\Big)^{\frac{1}{2}}
\\
& \leq  \int_{|x_0-z|>2r}|b_{B,w}-b(z)|^k \, \big\|\vec{f}(z)\big\|_{l_2} \, \frac{dz}{|x-z|^{n}}
\end{align*}
For $B(x)$, using the inequality $|z-x| \geq\frac{1}{2}|z-x_0|$, we have
\begin{align*}
B(x) & \lesssim \int_{|x_0-z|>2r} |b(z)-b_{B,w}|^k \, \big\|\vec{f}(z)\big\|_{l_2} \, \frac{dz}{|x_0-z|^{n}}
\\
& \lesssim \int_{|x_0-z|>2r}|b(z)-b_{B,w}|^k \, \big\|\vec{f}(z)\big\|_{l_2} \, \int_{|x_0-z|}^{\i}\frac{dt}{t^{n+1}}
\\
& \lesssim \int_{2r}^{\i}\int_{2r\leq |x_0-z|\leq t} |b(z)-b_{B,w}|^k \, \big\|\vec{f}(z)\big\|_{l_2} \, dz\frac{dt}{t^{n+1}}.
\end{align*}

Applying H\"older's inequality and by Lemma \ref{LinLU}, we get
\begin{align*}
\|B(\cdot)\|_{L_{w}^{p}(B)} & \lesssim w(B)^{\frac{1}{p}} \, \int_{2r}^{\i} \left(\int_{B(x_0,t)}|b(z)-b_{B,w}|^{kp'} w(z)^{1-p'}dz\right)^{\frac{1}{p'}} \|\|\vec{f}(\cdot)\|_{l_2}\|_{L^{p}_{w}(B(x_0,t))}\frac{dt}{t^{n+1}}
\\
& \lesssim \|b\|_{*} \, w(B)^{\frac{1}{p}} \, \int_{2r}^{\i}\Big(1+\ln^k \frac{t}{r}\Big) \, \|w^{-1/p}\|_{L_{p'}(B(x,t))} \,
\| \vec{f}\|_{L_{w}^{p}\big(B(x_0,t),l_2\big)} \, \frac{dt}{t^{n+1}}
\\
& \lesssim \|b\|_{*} w(B)^{\frac{1}{p}} \, \int_{2r}^{\i}\ln^k \Big(e+\frac{t}{r}\Big) \, \| \vec{f}\|_{L_{w}^{p}\big(B(x_0,t),l_2\big)} \, w(B(x_0,t))^{-1/p} \, \frac{dt}{t}.
\end{align*}

Thus,
$$
\big\|[b,\ G_{\alpha}]^k\vec{f}\big\|_{L_{w}^{p}\big(B,l_2\big)} \lesssim \|b\|_{*} \,  w(B)^{\frac{1}{p}} \, \int_{2r}^{\i}\ln^k \Big(e+\frac{t}{r}\Big) \, \| \vec{f}\|_{L_{w}^{p}\big(B(x_0,t),l_2\big)} \, w(B(x_0,t))^{-1/p} \, \frac{dt}{t}.
$$
\end{proof}
\

\textbf{Proof of Theorem \ref{wuteo1.3}}

By substitution of variables, we obtain
\begin{align*}
&\|[b,G_{\alpha}]^k\vec{f}\|_{M^{p,\varphi_2}_{w}(l_2)}
\\
&\lesssim \|b\|_{*} \, \sup_{x_0\in \Rn,r>0}\varphi_{2}(x_0,r)^{-1} \, \int_{2r}^{\i}\ln^k \Big(e+\frac{t}{r}\Big) \, \| \vec{f}\|_{L_{w}^{p}\big(B(x_0,t),l_2\big)} \, w(B(x_0,t))^{-1/p} \, \frac{dt}{t}
\\
& \lesssim \|b\|_{*} \, \sup_{x_0\in \Rn,r>0}\varphi_{2}(x_0,r)^{-1} \, \int_0^{r^{-1}} \ln^k \Big(e+\frac{1}{tr}\Big) \, \| \vec{f}\|_{L_{w}^{p}\big(B(x_0,t^{-1}),l_2\big)} \, w(B(x_0,t^{-1}))^{-\frac{1}{p}} \, \frac{dt}{t}
\\
&=\sup_{x\in\Rn,\,r>0} \|b\|_{*} \, \varphi_2(x_0,r^{-1})^{-1} \, r \, \frac{1}{r} \int_0^{r} \ln^k \Big(e+\frac{r}{t}\Big) \, \| \vec{f}\|_{L_{w}^{p}\big(B(x_0,t^{-1}),l_2\big)} \, w(B(x_0,t^{-1}))^{-\frac{1}{p}} \, \frac{dt}{t}
\\
& \lesssim \|b\|_{*} \, \sup_{x_0\in\Rn, r>0}\varphi_1(x_0,r^{-1})^{-1} w(B(x_0,r^{-1}))^{-\frac{1}{p}} \, \| \vec{f}\|_{L_{w}^{p}\big(B(x_0,r^{-1}),l_2\big)}
\\
& = \|b\|_{*} \, \sup_{x_0\in\Rn, r>0}\varphi_1(x_0,r)^{-1} w(B(x_0,r))^{-\frac{1}{p}} \, \big\|\vec{f}\big\|_{L^{p}_{w}(B(x_0,r),l_2)}
\\
& = \|b\|_{*} \, \|\vec{f}\|_{M^{p,\varphi_1}_{w}(l_2)}.
\end{align*}

By using the argument as similar as the above proofs and that of Theorem \ref{wuteo1.2}, we can also show the boundedness of $[b,\mathrm{g}_{\lambda,\alpha}^{*}]^k$.

\

\

\

Vagif S. Guliyev \\
Ahi Evran University, Department of Mathematics \\
Kirsehir, Turkey and \\
Institute of Mathematics  and Mechanics \\
Academy of Sciences of Azerbaijan \\
F. Agayev St. 9, Baku, AZ 1141, Azerbaijan

\

Mehriban N. Omarova \\
Baku State University \\
Baku, AZ 1148, Azerbaijan


\begin{thebibliography}{99}

\bibitem{AkbGulMust} Akbulut A., Guliyev V.S. and Mustafayev R.:
On the boundedness of the maximal operator and singular integral operators in generalized Morrey spaces, Math. Bohem. \textbf{137} (1), 27-43 (2012).

\bibitem{Cald1} Calderon A.P.: Commutators of singular integral operators, Proc. Natl. Acad. Sci. USA \textbf{53}, 1092-1099 (1965).

\bibitem{Cald2} Calderon A.P.: Cauchy integrals on Lipschitz curves and related operators, Proc. Natl. Acad. Sci. USA \textbf{74} (4), 1324-1327 (1977).

\bibitem{CarPickSorStep} Carro M., Pick L., Soria J., Stepanov V D.: On embeddings between classical Lorentz spaces, Math. Inequal. Appl. \textbf{4}, 397-428 (2001).

\bibitem{Chen} Chen Y.: Regularity of solutions to elliptic equations with VMO coefficients, Acta Math. Sin. (Engl. Ser.) \textbf{20}, 1103-1118 (2004).

\bibitem{ChFra} Chiarenza F., Frasca M.: Morrey spaces and Hardy-Littlewood maximal function, Rend Mat. \textbf{7}, 273-279 (1987).

\bibitem{ChFraL1} Chiarenza F., Frasca M., Longo P.: Interior $W^{2,p}$-estimates for nondivergence elliptic equations with discontinuous coefficients,
Ricerche Mat. \textbf{40}, 149-168 (1991).

\bibitem{ChFraL2} Chiarenza F., Frasca M., Longo P.: $W^{2,p}$-solvability of Dirichlet problem for nondivergence elliptic equations with VMO coefficients,
Trans. Amer. Math. Soc. \textbf{336}, 841-853 (1993).

\bibitem{CRW} Coifman R., Rochberg R., Weiss G.: Factorization theorems for Hardy spaces in several variables, Ann. of Math. \textbf{103} (2), 611-635 (1976).

\bibitem{FazRag2} Fazio G. Di, Ragusa M.A.: Interior estimates in Morrey spaces for strong solutions to nondivergence form equations with discontinuous coefficients,
J. Funct. Anal. \textbf{112}, 241-256 (1993).

\bibitem{FanLuY} Fan D., Lu S. and Yang D.: Boundedness of operators in Morrey spaces on homogeneous spaces and its applications, Acta Math. Sinica (N. S.) \textbf{14}, suppl., 625-634 (1998).

\bibitem{HuangLiu} Huang J.Z., Liu Y.: Some characterizations of weighted Hardy spaces, J. Math. Anal. Appl. \textbf{363}, 121-127 (2010).

\bibitem{DerGulSam} Deringoz, F., Guliyev, V.S., Samko, S.: Boundedness of maximal and singular operators on generalized Orlicz-Morrey spaces.
accepted in Advances in Harmonic Analysis and Operator Theory, Series: Operator Theory: Advances and Applications, Vol. \textbf{235}, 1-24 (2014).

\bibitem{DingLuY} Ding Y., Lu S.Z. and Yabuta K.: On commutators of Marcinkiewicz integrals with rough kernel, J. Math. Anal. Appl. \textbf{275}, 60-68 (2002).

\bibitem{Gi} Giaquinta M.: Multiple integrals in the calculus of variations and nonlinear elliptic systems. Princeton Univ. Press, Princeton, NJ, 1983.

\bibitem{Grafakos} Grafakos L.: Classical and Modern Fourier Analysis. Pearson Education, Inc. Upper Saddle River, New Jersey, 2004.

\bibitem{GulDoc}  Guliyev, V.S.: Integral operators on function spaces on the homogeneous groups and on domains in $\Rn$.
Doctor's degree dissertation, Mat. Inst. Steklov, Moscow,  329 pp. (in Russian) (1994).

\bibitem{GulJIA} Guliyev, V.S.: Boundedness of the maximal, potential and singular operators in the generalized Morrey spaces,
J. Inequal. Appl. Art. ID 503948 (2009). 20 pp.

\bibitem{GulAlKar1}  Guliyev, V.S., Aliyev S.S., Karaman T. : Boundedness of sublinear operators and commutators on generalized Morrey spaces,
Abstr. Appl. Anal. 2011, Art. ID 356041, 18 pp.

\bibitem{GULAKShIEOT2012} Guliyev, V.S., Aliyev, S.S., Karaman, T., Shukurov, P.S.:
Boundedness of sublinear operators and commutators on generalized Morrey Space. Int. Eq. Op. Theory. \textbf{71} (3), 327-355 (2011).

\bibitem{GulOMTSA} Guliyev, V.S.: Boundedness of classical operators and commutators of real analysis in generalized weighted Morrey spaces. Some applications,
International conference in honour of Professor V.I. Burenkov on the occasion of his 70th birthday to be held in Kirsehir, Turkey, from May 20 to may 27, 2011.

\bibitem{GulEMJ2012} Guliyev, V.S.: Generalized weighted Morrey spaces and higher order commutators of sublinear operators, Eurasian Math. J. \textbf{3}(3), 33-61 (2012).

\bibitem{GulSoftPot} Guliyev, V.S., L. Softova,
Global regularity in generalized Morrey spaces of solutions to nondivergence elliptic equations with VMO coefficients, Potential Anal. 38 (4) 2013, 843-862.

\bibitem{GulSoftPrEd} Guliyev, V.S., L. Softova,
Generalized Morrey regularity for parabolic equations with discontinuity data, Proc. Edinb. Math. Soc. (in press).


\bibitem{GulKarMustSer} Guliyev, V.S., Karaman, T., Mustafayev R.Ch., Serbetci  A.,
Commutators of sublinear operators generated by Calder\'{o}n-Zygmund operator on generalized weighted Morrey spaces, Czechoslovak Math. J. (in press).

\bibitem{GulShuk2013} Guliyev, V.S., Shukurov, P.S. : Commutators of intrinsic square functions on generalized Morrey spaces, Proceedings of IMM of NAS of Azerbaijan. (in press).

\bibitem{GulMathN} Guliyev, V.S., : Commutators of intrinsic square functions on generalized weighted Morrey spaces, submitted.

\bibitem{KarGulSer} Karaman, T., Guliyev, V.S., Serbetci  A.,
Boundedness of sublinear operators generated by Calder\'{o}n-Zygmund operators on generalized weighted Morrey spaces,
Scientic Annals of "Al.I. Cuza" University of Iasi, \textbf{60} (1), 1-18 2014. ~~  DOI: 10.2478/aicu-2013-0009

\bibitem{KomShir} Y. Komori and S. Shirai, Weighted Morrey spaces and a singular integral operator, Math. Nachr. (2) 282 (2009), 219-231.

\bibitem{KJF} Kufner A., John O. and Fu\c{c}ik S.:Function Spaces. Noordhoff International Publishing: Leyden, Publishing House Czechoslovak Academy of Sciences: Prague, 1977.

\bibitem{Lerner} Lerner A.K.: Sharp weighted norm inequalities for Littlewood-Paley operators and singular integrals, Adv. Math. \textbf{226}, 3912-3926 (2011).

\bibitem{Morrey} Morrey, C.B.: On the solutions of quasi-linear elliptic partial differential equations. Trans. Amer. Math. Soc. \textbf{43}, 126-166 (1938).

\bibitem{Muckenh} Muckenhoupt B.: Weighted norm inequalities for the Hardy maximal function, Trans. Amer. Math. Soc. \textbf{165}, 207-226 (1972).

\bibitem{MuckWh1} Muckenhoupt B. and Wheeden R.: Weighted norm inequalities for fractional integrals, Trans. Amer. Math. Soc. \textbf{192}, 261-274 (1974).

\bibitem{MustAJM} Mustafayev R.Ch.: On boundedness of sublinear operators in weighted Morrey spaces, Azerb. J. Math. \textbf{2} (1), 66-79 (2012).

\bibitem{Wang1W} Wang, H.: Weak type estimates for intrinsic square functions on weighted Morrey spaces. Anal. Theory Appl. \textbf{29} (2), 104-119 (2013).

\bibitem{Wang1} Wang, H.: Intrinsic square functions on the weighted Morrey spaces. J. Math. Anal. Appl. \textbf{396},  302-314 (2012).

\bibitem{Wang2} Wang, H.: Boundedness of intrinsic square functions on the weighted weak Hardy spaces. Integr. Equ. Oper. Theory \textbf{75}, 135-149 (2013).

\bibitem{WangLiu} Wang, H., Liu, H. P.: Weak type estimates of intrinsic square functions on the weighted Hardy spaces. Arch. Math. \textbf{97}, 49-59 (2011).

\bibitem{Wilson1} Wilson, M.: The intrinsic square function. Rev. Mat. Iberoam. \textbf{23}, 771-791 (2007).

\bibitem{Wilson2} Wilson, M.: Weighted Littlewood-Paley theory and Exponential-square integrability. Lecture Notes in Math. vol. \textbf{1924}, Springer-Verlag (2007).

\end{thebibliography}
\end{document}